\documentclass[a4paper]{amsart}
\usepackage{amsmath, amsthm, amsfonts, mathtools, tikz-cd, amssymb}
\usepackage{enumitem}
\usepackage{hyperref}

\title{Combination of open covers with $\pi_1$-constraints}
\date{May 7, 2025.\ \copyright{\ P.~Capovilla, K.~Li, C.~L\"oh 2025}.
  This work was partially supported by the CRC~1085 \emph{Higher Invariants}
  (Universit\"at Regensburg, funded by the DFG)}

\author[P.~Capovilla]{Pietro Capovilla}
\address{Scuola Normale Superiore, 56126 Pisa, Italy}
\email{pietro.capovilla@sns.it}

\author[K.~Li]{Kevin Li}
\address{Fakult\"at f\"ur Mathematik, Universit\"at Regensburg, 93040 Regensburg, Germany}
\email{kevin.li@ur.de}

\author[C.~L\"oh]{Clara L\"oh}
\address{Fakult\"at f\"ur Mathematik, Universit\"at Regensburg, 93040 Regensburg, Germany}
\email{clara.loeh@ur.de}

\theoremstyle{definition}
\newtheorem{defn}{Definition}[section]
\newtheorem{ex}[defn]{Example}
\newtheorem{rem}[defn]{Remark}
\newtheorem{setup}[defn]{Setup}
\theoremstyle{plain}
\newtheorem{thm}[defn]{Theorem}
\newtheorem{lem}[defn]{Lemma}
\newtheorem{prop}[defn]{Proposition}
\newtheorem{cor}[defn]{Corollary}

\numberwithin{equation}{section}
\newcommand{\enum}{\rm{(\roman*)}}

\newcommand{\IN}{\ensuremath{\mathbb{N}}}
\newcommand{\IR}{\ensuremath{\mathbb{R}}}
\newcommand{\IZ}{\ensuremath{\mathbb{Z}}}

\newcommand{\calD}{\ensuremath{\mathcal{D}}}
\newcommand{\calF}{\ensuremath{\mathcal{F}}}
\newcommand{\calQ}{\ensuremath{\mathcal{Q}}}
\newcommand{\calS}{\ensuremath{\mathcal{S}}}

\newcommand{\AME}{\ensuremath{\mathsf{Am}}}
\newcommand{\FIN}{\ensuremath{\mathsf{Fin}}}
\newcommand{\TR}{\ensuremath{\mathsf{Tr}}}

\DeclareMathOperator{\cat}{cat}
\DeclareMathOperator{\gd}{gd}
\DeclareMathOperator{\cd}{cd}
\DeclareMathOperator{\vcd}{vcd}
\DeclareMathOperator{\TC}{TC}
\DeclareMathOperator{\im}{im}
\DeclareMathOperator{\id}{id}
\DeclareMathOperator{\Cyl}{Cyl}
\DeclareMathOperator{\pr}{pr}

\newcommand{\brW}[1]{\ensuremath{W\!\langle{#1}\rangle}}
\newcommand{\catres}[2]{\ensuremath{\cat_{#1|_{#2}}\!(#2)}}

\begin{document}

\begin{abstract}
	Let~$G$ be a group and let~$\calF$ be a family of subgroups of~$G$.
	The generalised Lusternik--Schnirelmann category~$\operatorname{cat}_\mathcal{F}(G)$ is the minimal cardinality of covers of~$BG$ by open subsets with fundamental group in~$\calF$.
	We prove a combination theorem for~$\operatorname{cat}_\mathcal{F}(G)$ in terms of the stabilisers of contractible $G$-CW-complexes.
	As applications for the amenable category, we obtain vanishing results for the simplicial volume of gluings of manifolds (along not necessarily amenable boundaries) and of cyclic branched coverings.
	Moreover, we deduce an upper bound for Farber's topological complexity, generalising an estimate for amalgamated products of Dranishnikov--Sadykov. 
\end{abstract}

\maketitle	
	
\section{Introduction}
Measures of complexity for discrete groups can be defined topologically via  classifying spaces.
A classical example is the minimal dimension. 
The \emph{geometric dimension~$\gd(G)$} of a group~$G$ is the infimum of~$n\in \IN$ for which there exists an $n$-dimensional CW-model for~$BG$.
A useful tool for estimating the geometric dimension is the following combination theorem.
For a cell~$\sigma$ of a $G$-CW-complex, we denote its stabiliser by $G_\sigma\coloneqq \{g\in G\mid g\sigma=\sigma\}$.

\begin{thm}[{\cite{Brown82,Geoghegan08}}]
\label{thm:combination_gd}
	Let~$G$ be a group and let~$X$ be a (non-equivariantly) contractible $G$-CW-complex.
	Let~$\Sigma$ be a set of $G$-orbit representatives of all cells of~$X$.
	Then
	\[
		\gd(G)\le \sup_{\sigma\in \Sigma} \bigl( \gd(G_\sigma)+\dim(\sigma)\bigr).
	\]
\end{thm}

The upper bound in Theorem~\ref{thm:combination_gd} is independent of the choice of orbit representatives and it is meaningful only if~$X$ is finite dimensional.

\subsection*{Open covers with $\pi_1$-constraints}
A different measure of complexity is given by the minimal cardinality (or multiplicity) of open covers with constraints on fundamental groups.
This notion is similar in spirit to the classical Lusternik--Schnirelmann category~\cite{CLOT03}.
Let~$G$ be a group.
A \emph{family~$\calF$ of subgroups of~$G$} is a non-empty set of subgroups of~$G$ that is closed under conjugation and taking subgroups.
Typical families are~$\TR$, $\FIN$, and~$\AME$ consisting of the trivial subgroup, all finite subgroups, and all amenable subgroups, respectively.

\begin{defn}
\label{defn:catF}
	Let~$Y$ be a path-connected (pointed) topological space and let~$\calF$ be a family of subgroups of~$\pi_1(Y)$.
	The \emph{$\calF$-category~$\cat_\calF(Y)$ of~$Y$} is defined as the infimum of~$n\in \IN$ for which there exists a cover of~$Y$ by (not necessarily path-connected) open subsets~$U_0,U_1,\ldots,U_n$ such that
	 \[
	 	\im\bigl(\pi_1(U_i,y)\to \pi_1(Y,y)\bigr)\in \calF
	 \]
	 for all~$i\in \{0,\ldots,n\}$ and all~$y\in U_i$.
	 
	 Let~$G$ be a group and let~$\calF$ be a family of subgroups of~$G$.
	 Then we define the \emph{$\calF$-category of~$G$} as $\cat_\calF(G)\coloneqq \cat_\calF(BG)$.
	 If there exists a $\pi_1$-isomorphism $Y\to Z$, then pulling back open covers shows $\cat_\calF(Y)\le \cat_\calF(Z)$.
	 In particular, $\cat_\calF(G)$ is well-defined and we have $\cat_\calF(Y)\le \cat_\calF(\pi_1(Y))$.
\end{defn}

We have $\cat_\calF(G)=0$ if and only if~$G\in \calF$.
Moreover, $\cat_\calF(G)=1$ if and only if~$G\notin \calF$ and~$G$ is the fundamental group of a graph of groups whose vertex groups (and hence also edge groups) lie in~$\calF$~\cite[Proposition~5.3]{CLM22}.
In the literature~\cite{CLM22,Loeh-Moraschini22} also a different normalisation is used, producing values that are larger by~$1$ than ours.

The $\calF$-category has many interesting special cases.
The trivial category~$\cat_\TR(G)$ equals the cohomological dimension~$\cd(G)$ for every group~$G$ by a classical result of Eilenberg--Ganea~\cite{EG57}, Stallings~\cite{Sta68}, and Swan \cite{Swa69}.
The finite category~$\cat_\FIN(G)$ of a virtually torsion-free group~$G$ is bounded from below by the virtual cohomological dimension~$\vcd(G)$ and equality holds, e.g., for right-angled Coxeter groups~\cite{Li23}.
The amenable category~$\cat_\AME(G)$ appears in vanishing results for $\ell^2$-Betti numbers~\cite{Sauer09} and (the comparison map in) bounded cohomology~\cite{Gromov82}. 
These homological invariants of~$G$ vanish in all degrees larger than~$\cat_\AME(G)$, in other words, they are obstructions to the existence of amenable covers. 
Moreover, for aspherical closed manifolds, small amenable category 
leads also to the vanishing of integral foliated simplicial volume and stable integral simplicial volume~\cite{Loeh-Moraschini-Sauer-22}. Conversely, small minimal volume 
implies small amenable category~\cite[Section~3.4]{Gromov82}. 
The $\calF$-category with respect to families~$\calF$ of subgroups with controlled growth is related to the (non-)vanishing of minimal volume entropy~\cite{Babenko-Sabourau21,Loeh-Moraschini22}.
The diagonal category~$\cat_\calD(G\times G)$, where~$\calD$ is the smallest family of subgroups of~$G\times G$ containing the diagonal subgroup, equals Farber's topological complexity~$\TC(G)$~\cite{FGLO19}.

\subsection*{Combination theorems}
We prove two combination theorems for the $\calF$-category of groups and, in fact, a common generalisation (Theorem~\ref{thm:main}).
For a family~$\calF$ of subgroups of~$G$ and a subgroup~$H$ of~$G$, we denote the restricted family of subgroups of~$H$ by $\calF|_H\coloneqq \{F\subset H\mid F\in \calF\}$.

\begin{thm}
\label{thm:main_max}
	Let~$G$ be a group and let~$X$ be a (non-equivariantly) contractible $G$-CW-complex.
	Let~$\Sigma_0$ and~$\Sigma_{\ge 1}$ be sets of $G$-orbit representatives of all $0$-cells and all positive dimensional cells of~$X$, respectively.
	Let~$\calF$ be a family of subgroups of~$G$.
	Then
	\[
		\cat_\calF(G)\le 
		\max\Bigl\{\sup_{v\in \Sigma_0} \catres{\calF}{G_v}, \sup_{\sigma\in \Sigma_{\ge 1}}\bigl(\gd(G_\sigma)+\dim(\sigma)\bigr)\Bigr\}.
	\]
\end{thm}

While Theorem~\ref{thm:main_max} is similar to Theorem~\ref{thm:combination_gd}, the following is analogous to a standard estimate for the Lusternik--Schnirelmann category of homotopy pushouts.

\begin{thm}
\label{thm:main_sum}
	Let~$G$ be a group and let~$X$ be a (non-equivariantly) contractible $G$-CW-complex.
	For every~$i\in \IN$, let~$\Sigma_i$ be a set of $G$-orbit representatives of all $i$-cells of~$X$.
	Let~$\calF$ be a family of subgroups of~$G$.
	Then
	\[
		\cat_\calF(G)\le \sup_{v\in \Sigma_0}\catres{\calF}{G_v}+\sum_{i\in \IN_{\ge 1}} \sup_{\sigma\in \Sigma_i}\bigl(\catres{\calF}{G_\sigma}+1\bigr).
	\]
\end{thm}

In the special case that~$X$ is $1$-dimensional, we obtain two upper bounds for the $\calF$-category of graphs of groups.

\begin{cor}
\label{cor:gog}
	Let~$G$ be the fundamental group of a graph of groups with vertex groups~$(G_v)_{v\in V}$ and edge groups~$(G_e)_{e\in E}$.
	Let~$\calF$ be a family of subgroups of~$G$.
	Then the following hold:
	\begin{enumerate}[label=\enum]
	\item\label{item:gog_max}
	$\cat_\calF(G)\le \max\bigl\{\sup_{v\in V}\catres{\calF}{G_v},\, \sup_{e\in E}\bigl(\gd(G_e)+1\bigr)\bigr\}$;
	\item\label{item:gog_sum}
	$\cat_\calF(G)\le \sup_{v\in V}\catres{\calF}{G_v}+\sup_{e\in E}\bigl(\catres{\calF}{G_e}+1\bigr)$.
    \qed
	\end{enumerate}
\end{cor}

Already for amalgamated products, neither of the two estimates in Corollary~\ref{cor:gog} is better than the other in general.
For example, we have $\cat_\TR(\IZ\ast \IZ)=1$ and $\cat_\FIN(\IZ/4\ast_{\IZ/2} \IZ/6)=1$.
Moreover, in general $\cat_\calF(G_1\ast_{G_0} G_2)$ is not bounded above by $\max\{\catres{\calF}{G_1},\catres{\calF}{G_2},\catres{\calF}{G_0}+1\}$ as shown by the example $\cat_\AME(F_2\ast_{C} F_2)=2$ (e.g., \cite[Example~7.8]{CLM22}), where~$F_2$ is the free group freely generated by~$a$ and~$b$, and~$C$ is the infinite cyclic group generated by the commutator~$[a,b]$.

\subsection*{Applications}
Let~$M$ be an oriented closed connected manifold.
The \emph{simplicial volume~$\|M\|$ of~$M$} is the infimum of $\ell^1$-norms of real fundamental cycles~\cite{Gromov82}.
One of the most widely applicable vanishing results for simplicial volume is Gromov's vanishing theorem.

\begin{thm}[{\cite[p.~41]{Gromov82}}]
\label{thm:sv_vanishing}
	Let~$M$ be an oriented closed connected manifold.
	If $\cat_\AME(M)<\dim(M)$, then $\|M\|=0$.
\end{thm}

As an application of Corollary~\ref{cor:gog} for the amenable category of graphs of groups, we obtain vanishing results for the simplicial volume of gluings of manifolds with boundary (Theorem~\ref{thm:cat_gluing}) and, in particular, for twisted doubles (Corollary~\ref{cor:cat_double}).
While most gluing estimates in the literature~\cite{Gromov82,BBFIPP14,Kue15,LLM22,Cap24} require amenability of the boundaries, our results apply to non-amenable boundaries as well.

The combination theorem for the amenable category of polygons of groups (Corollary~\ref{cor:polygon}) yields a vanishing result for the simplicial volume of cyclic branched coverings (Theorem~\ref{thm:cat_branching}).

We also deduce a combination theorem for topological complexity (Theorem~\ref{thm:TC}), generalising a result for amalgamated products of Dranishnikov--Sadykov~\cite{Dranishnikov-Sadykov19}.

\section{Combination of open covers via $G$-spaces}
We prove a combination theorem for the $\calF$-category (Theorem~\ref{thm:main}) that contains Theorem~\ref{thm:main_max} and Theorem~\ref{thm:main_sum} as special cases.
We work with an equivalent characterisation of the $\calF$-category phrased in terms of the dimension of $G$-CW-complexes with isotropy in~$\calF$.
A $G$-CW-complex \emph{has isotropy in~$\calF$} if all stabilisers of cells lie in~$\calF$. Then, instead of constructing open covers explicitly, we estimate the dimension of $G$-pushouts appearing in appropriate constructions of classifying spaces.

\begin{defn}
	Let~$G$ be a group and let~$\calF$ be a family of subgroups of~$G$.
	Let~$X$ be a $G$-CW-complex.
	We define~$\cat^G_\calF(X)$ as the infimum of~$n\in \IN$ for which there exist an $n$-dimensional $G$-CW-complex~$N$ with isotropy in~$\calF$ and a $G$-map $X\to N$.
\end{defn}

If there exists a $G$-map $X_1\to X_2$, then $\cat^G_\calF(X_1)\le \cat^G_\calF(X_2)$.
In particular, $\cat^G_\calF$ is invariant under $G$-homotopy equivalence.
The following basic properties are immediate.

\begin{lem}
\label{lem:catGF_basic}
	Let~$G$ be a group and let~$\calF$ be a family of subgroups of~$G$.
	\begin{enumerate}[label=\enum]
		\item Let~$(X_i)_{i\in I}$ be a collection of $G$-CW-complexes.
		Then 
		\[
			\cat^G_\calF\Bigl(\coprod_{i\in I} X_i\Bigr)=\sup_{i\in I}\cat^G_\calF(X_i);
		\]
		\item Let~$H$ be a subgroup of~$G$ and let~$Z$ be an $H$-CW-complex.
		Then 
		\[
            \pushQED{\qed}
			\cat^G_\calF(G\times _H Z)\le \cat^H_{\calF|_H}(Z).
            \qedhere
            \popQED
		\]
	\end{enumerate}
\end{lem}

The $\calF$-category of a space (Definition~\ref{defn:catF}) equals~$\cat^{\pi_1}_\calF$ of its universal covering.
This is proved using equivariant nerve maps and covers of simplicial complexes by open stars.

\begin{prop}[{\cite[Lemma~7.6]{CLM22}}]
\label{prop:CLM}
	Let~$Y$ be a connected CW-complex with fundamental group~$G$ and let~$\widetilde{Y}$ be the universal covering of~$Y$ (with the induced $G$-CW-structure).
	Let~$\calF$ be a family of subgroups of~$G$.
	Then
	\[
		\cat_\calF(Y)=\cat^G_\calF(\widetilde{Y}).
	\]
	In particular, $\cat_\calF(G)=\cat^G_\calF(EG)$.
\end{prop}

The next lemma is crucial.
It provides two different estimates for~$\cat^G_\calF$ of certain $G$-pushouts.

\begin{lem}
\label{lem:key}
	Let~$G$ be a group and let~$\calF$ be a family of subgroups of~$G$.
	Let $n\in \IN_{\ge 1}$.
	Consider a $G$-pushout of~$G$-CW-complexes with isotropy in~$\calF$ of the form
	\[\begin{tikzcd}
		X_0\times S^{n-1}\ar{r}{f}\ar[hook]{d}
		& X_1\ar{d}
		\\
		X_0\times D^n\ar{r}
		& X
	\end{tikzcd}\]
	where~$f$ is a $G$-map.
	Then the following hold:
	\begin{enumerate}[label=\enum]
		\item\label{item:key_max} 
		$\cat^G_\calF(X)\le \max\{\cat^G_\calF(X_1), \dim(X_0)+n\}$;
		\item\label{item:key_sum}
		$\cat^G_\calF(X)\le \cat^G_\calF(X_1)+\cat^G_\calF(X_0)+1$.
	\end{enumerate}
	\begin{proof}
		For~$i\in\{0,1\}$, let~$\nu_i\colon X_i\to N_i$ be a $G$-map, where~$N_i$ is a $G$-CW-complex with isotropy in~$\calF$. 
		In both cases~(i) and~(ii), we construct a suitable $G$-CW-complex~$N$ with isotropy in~$\calF$ as a $G$-pushout. 
		Then the universal property of the $G$-pushout~$X$ yields a $G$-map~$\nu\colon X\to N$, showing that $\cat^G_\calF (X) \leq \dim (N)$.
		
		(i) We consider the following commutative diagram, where $N$ is defined as the pushout of the outer diagram and $\nu \colon X \to N$ is obtained from the inner pushout:
		\[\begin{tikzcd}
			X_0\times S^{n-1}\ar{rrr}{\nu_1\circ f}\ar[hook]{ddd}
			&&& N_1\ar[dashed]{ddd}
			\\
			& X_0\times S^{n-1}\ar{r}{f}\ar[hook]{d}\ar{ul}{\id}
			& X_1\ar{d}\ar{ur}[swap]{\nu_1}
			\\
			& X_0\times D_n\ar{r}\ar{dl}[swap]{\id}
			& X\ar[dashed]{dr}{\nu}
			\\
			X_0\times D^n\ar[dashed]{rrr}
			&&& N
		\end{tikzcd}\]
		By equivariant cellular approximation, we may assume that the $G$-maps~$f$ and~$\nu_1$ are cellular, so that~$N$ inherits the structure of a $G$-CW-complex.
		By construction, the isotropy groups of~$N$ are contained in~$\calF$ and
		we have 
		\[
			\dim(N)=\max\{\dim(N_1), \dim(X_0)+n\}.
		\]
		
		(ii) Let~$\Cyl(\pr_{N_0})$ be the mapping cylinder of the projection $\pr_{N_0}\colon N_0\times N_1\to N_0$ (with the induced $G$-CW-structure).
		We consider the commutative diagram, where the $G$-CW-complex~$N$ is defined as the pushout of the outer diagram and $\nu \colon X \to N$ is obtained from the inner pushout:
		\[\begin{tikzcd}
			N_0\times N_1\ar{rrr}{\pr_{N_1}}\ar[hook]{ddd}
			&&& N_1\ar[dashed]{ddd}
			\\
			& X_0\times S^{n-1}\ar{r}{f}\ar[hook]{d}\ar{ul}{\varphi}
			& X_1\ar{d}\ar{ur}[swap]{\nu_1}
			\\
			& X_0\times D_n\ar{r}\ar{dl}[swap]{\Phi}
			& X\ar[dashed]{dr}{\nu}
			\\
			\Cyl(\pr_{N_0})\ar[dashed]{rrr}
			&&& N
		\end{tikzcd}\]
        Here $\varphi\coloneqq (\nu_0\circ \pr_{X_0})\times (\nu_1\circ f)$ and the $G$-map~$\Phi$ is induced by the inner pushout of the commutative diagram
		\[\begin{tikzcd}
			N_0\times N_1\times \{1\}\ar{rrr}{\pr_{N_0}}\ar[hook]{ddd}
			&&& N_0\ar{ddd}
			\\
			& X_0\times S^{n-1}\times \{1\}\ar{r}{\pr_{X_0}}\ar[hook]{d}\ar{ul}{\varphi}
			& X_0\ar{d}\ar{ur}[swap]{\nu_0}
			\\
			& X_0\times S^{n-1}\times [0,1]\ar{r}\ar{dl}[swap]{\varphi\times \id_{[0,1]}}
			& X_0\times D^n\ar[dashed]{dr}{\Phi}
			\\
			N_0\times N_1\times [0,1]\ar{rrr}
			&&& \Cyl(\pr_{N_0})
		\end{tikzcd}\]
		By construction, the isotropy groups of~$N$ are contained in~$\calF$ and we have
		\[
			\dim(N)= \dim(N_1)+\dim(N_0)+1.
		\]
		This finishes the proof.
	\end{proof}
\end{lem}

We will use a well-known ``blow-up" construction~\cite[Proof of Theorem~3.1]{Lueck00}, see also~\cite[Theorem~2.3]{MPSS20}.

\begin{thm}[{\cite{Lueck00}}]
\label{thm:blow-up}
	Let~$G$ be a group and let~$X$ be a $G$-CW-complex.
	For every~$i\in \IN$, let~$\Sigma_i$ be a set of $G$-orbit representatives of all $i$-cells of~$X$.
	For every~$\sigma\in \Sigma_i$, we fix a model for~$EG_\sigma$.
	Then there exists a free $G$-CW-complex~$\widehat{X}$ and a $G$-map $p\colon \widehat{X}\to X$ satisfying the following:
	\begin{enumerate}[label=\enum]
		\item\label{item:blow-up_pushout} 
		$\widehat{X}$ admits a filtration~$(\widehat{X}_i)_{i\in \IN}$ such that there exist $G$-pushouts of the form
		\[\begin{tikzcd}
			\coprod_{\sigma\in \Sigma_i} G\times_{G_\sigma} EG_\sigma\times S^{i-1}\ar{r}\ar{d}
			& \widehat{X}_{i-1}\ar{d}
			\\
			\coprod_{\sigma\in \Sigma_i} G\times_{G_\sigma} EG_\sigma\times D^i\ar{r}
			& \widehat{X}_i
		\end{tikzcd}\]
		In particular, $\widehat{X}_0=\coprod_{v\in \Sigma_0} G\times_{G_v} EG_v\times D^0$;
		\item $p$ is a (non-equivariant) homotopy equivalence;
		\item $\widehat{X}$ is $G$-homotopy equivalent to~$EG\times X$ equipped with the diagonal $G$-action.
	\end{enumerate}
\end{thm}

If~$X$ is contractible, then~$\widehat{X}$ is a model for~$EG$. 
Hence Theorem~\ref{thm:combination_gd} is an immediate consequence of Theorem~\ref{thm:blow-up}.

The following is our main result.
For each dimension~$i\in \IN_{\ge 1}$, we apply either Lemma~\ref{lem:key}~\ref{item:key_max} or~\ref{item:key_sum} to the $G$-pushout in Theorem~\ref{thm:blow-up}~\ref{item:blow-up_pushout}.

\begin{thm}
\label{thm:main}
	Let~$G$ be a group and let~$\calF$ be a family of subgroups of~$G$.
	Let~$X$ be a $G$-CW-complex.
	For every~$i\in \IN$, let~$\Sigma_i$ be a set of $G$-orbit representatives of all $i$-cells of~$X$.
	Let~$I$ be a subset of~$\IN_{\ge 1}$.
	For~$i\in \IN_{\ge 1}$, we set
	\begin{align*}
		d_0
		&\coloneqq \sup_{v\in \Sigma_0}\catres{\calF}{G_v};
		\\
		d_i
		&\coloneqq \begin{cases}
			\max\bigl\{d_{i-1}, \, \sup_{\sigma\in \Sigma_i}\bigl(\gd(G_\sigma)+i\bigr)\bigr\}
			& \text{if } i\in I;
			\\
			d_{i-1}+\sup_{\sigma\in \Sigma_i}\bigl(\catres{\calF}{G_\sigma}+1\bigr)
			& \text{if } i\notin I.
		\end{cases}
	\end{align*}
	If~$X$ is $n$-dimensional,
	then $\cat^G_\calF(EG\times X)\le d_n$.

        In particular, if~$X$ is $n$-dimensional and (non-equivariantly) contractible, 
        then $\cat_\calF (G) \leq d_n$.
	\begin{proof}
		For every~$i\in \IN$ and every $\sigma\in \Sigma_i$, we fix a model for~$EG_\sigma$ of dimension~$\gd(G_\sigma)$.
		Theorem~\ref{thm:blow-up} yields a free $G$-CW-complex~$\widehat{X}$ that is $G$-homotopy equivalent to~$EG\times X$.
		We proceed by induction over the filtration~$(\widehat{X}_i)_{i\in \IN}$ of~$\widehat{X}$.
		Since $\widehat{X}_0=\coprod_{v\in \Sigma_0} G\times_{G_v} EG_v\times D^0$, we have $\cat^G_\calF(\widehat{X}_0)\le d_0$ by Lemma~\ref{lem:catGF_basic}.
		For every~$i\in \IN_{\ge 1}$, there exists a $G$-pushout
		\[\begin{tikzcd}
			\coprod_{\sigma\in \Sigma_i} G\times_{G_\sigma} EG_\sigma\times S^{i-1}\ar{r}\ar{d}
			& \widehat{X}_{i-1}\ar{d}
			\\
			\coprod_{\sigma\in \Sigma_i} G\times_{G_\sigma} EG_\sigma\times D^i\ar{r}
			& \widehat{X}_i
		\end{tikzcd}\]
		By induction, we have $\cat^G_\calF(\widehat{X}_{i-1})\le d_{i-1}$.
		If~$i\in I$, Lemma~\ref{lem:key}~\ref{item:key_max} yields
		\[
			\cat^G_\calF(\widehat{X}_i)\le 
			\max\Bigl\{d_{i-1}, \sup_{\sigma\in \Sigma_i}\bigl(\gd(G_\sigma)+i\bigr)\Bigr\} = d_i.
		\]
		If~$i\notin I$, Lemma~\ref{lem:key}~\ref{item:key_sum} yields
		\[
			\cat^G_\calF(\widehat{X}_i)\le
			d_{i-1}+\sup_{\sigma\in \Sigma_i}\bigl(\catres{\calF}{G_\sigma}+1\bigr) = d_i.
		\]
		If $X$ is contractible, then $\widehat X$ is a model of~$EG$ and 
            hence $\cat_\calF (G) = \cat^G_\calF (\widehat X)$.
	\end{proof}
\end{thm}

\begin{proof}[Proof of Theorem~\ref{thm:main_max}]
	This follows from Theorem~\ref{thm:main} for~$I=\IN_{\ge 1}$.
\end{proof}

\begin{proof}[Proof of Theorem~\ref{thm:main_sum}]
	This follows from Theorem~\ref{thm:main} for~$I=\emptyset$.
\end{proof}

\section{Simplicial volume of gluings}
We prove vanishing results for the simplicial volume of gluings of manifolds with boundary by applying our combination theorems to the amenable category. 
The category estimates below are not specific to the amenable family, but in view of the applications to simplicial volume, we treat only this case. 
For background on simplicial volume, we refer the reader to the literature~\cite{Gromov82,Frigerio17}.

\begin{defn}
	Let~$(M,\partial M)$ be an oriented compact connected $n$-manifold with (possibly empty) boundary.
	The \emph{simplicial volume~$\|M,\partial M\|\in \IR_{\ge 0}$ of~$(M,\partial M)$} is defined as
	\[
		\|M,\partial M\|\coloneqq \inf \bigl\{|c|_1 \bigm\vert c\in C_n(M;\IR) \text{ is a relative fundamental cycle of~$(M,\partial M)$}\bigr\}.	
	\]
	Here the $\ell^1$-norm of a singular chain $c=\sum_\sigma a_\sigma\cdot \sigma\in C_n(M;\IR)$
	is defined as $|c|_1\coloneqq \sum_\sigma |a_\sigma|$.
\end{defn}

The simplicial volume is zero in the presence of small amenable covers.
The case of closed manifolds is Theorem~\ref{thm:sv_vanishing}. We extend the argument to the relative case.  
            
\begin{lem}
\label{lem:sv_vanishing}
	Let~$(M,\partial M)$ be an oriented compact connected $n$-manifold.
	Suppose that the following hold:
	\begin{enumerate}[label=\enum]
		\item Every boundary component of~$M$ is $\pi_1$-injective in~$M$;
		\item Every boundary component of~$M$ has fundamental group of geometric dimension~$\le n-2$;
		\item $\cat_\AME(\pi_1(M))\le n-1$.
	\end{enumerate}
	Then $\|M,\partial M\|=0$.
	\begin{proof} 
		By the duality principle~\cite[Proposition~7.10]{Frigerio17}, the simplicial volume of~$(M,\partial M)$ is zero if and only if the comparison map in degree~$n$
		\[
			H_b^n(M,\partial M)\to H^n(M,\partial M)
		\]
		vanishes.
		Here~$H_b^*$ and~$H^*$ denote bounded cohomology and singular cohomology with real coefficients, respectively.
		Let~$S_1, \ldots, S_l$ be the boundary components of~$M$.
		Let~$\pi_1(\calS)$ denote the collection $(\pi_1(S_i))_i$ of subgroups of~$\pi_1(M)$.
		By the mapping theorem~\cite[Theorem~5.9]{Frigerio17} and the five lemma, it suffices to show that the comparison map for the group pair~$(\pi_1(M),\pi_1(\calS))$ vanishes in degree~$n$.
		We consider the following diagram
		\[\begin{tikzcd}
			\hspace{4ex} \prod_{i=1}^l H_b^{n-1}(\pi_1(S_i))\ar{r}\ar{d}
			& H_b^n(\pi_1(M),\pi_1(\calS))\ar{r}\ar{d}
			& H_b^n(\pi_1(M))\ \ar{d}{0}
			\\
			0=\prod_{i=1}^l H^{n-1}(\pi_1(S_i))\ar{r}
			& H^n(\pi_1(M),\pi_1(\calS))\ar{r}
			& H^n(\pi_1(M))
		\end{tikzcd}\]
		where the rows are portions of the long exact sequences of pairs and the vertical maps are comparison maps.
		Since $\gd(\pi_1(S_i))\le n-2$ for all~$i\in \{1,\ldots,l\}$, the lower left module is trivial.
		Since $\cat_\AME(\pi_1(M))\le n-1$, the right vertical map vanishes by the vanishing theorem~\cite[p.~40]{Gromov82}.
		It follows that the middle vertical map vanishes, as desired.
	\end{proof}
\end{lem}

We fix some notation for gluings of (topological) manifolds with boundary.

\begin{setup}
\label{setup:gluing}
	Let~$n, k , l \in \IN$.
	\leavevmode
	\begin{itemize}
		\item Let~$M_1, \ldots, M_k$ be oriented compact connected $n$-manifolds.
		\item We fix pairings~$(S_1^+,S_1^-), \ldots, (S_l^+,S_l^-)$ of some boundary components of the~$(M_j)_j$ such that every boundary component of the~$(M_j)_j$ occurs at most once among the~$(S_i^\pm)_i$.
		\item For all~$i\in \{1,\ldots,l\}$, let~$j^\pm(i)\in \{1,\ldots,k\}$ be the above index such that $S_i^\pm$ is a boundary component of~$M_{j^\pm(i)}$.
		\item For all~$i\in \{1,\ldots,l\}$, let~$f_i\colon S_i^+\to S_i^-$ be an orientation-reversing homeomorphism.
		\item Let~$M$ be the oriented compact $n$-manifold obtained by gluing~$M_1, \ldots, M_k$ along~$f_1, \ldots, f_l$.
	We assume that~$M$ is connected.
	\end{itemize}
\end{setup}

If all boundary components are amenable, then the simplicial volume is subadditive.
In particular, the simplicial volume of the gluing is zero if the simplicial volume of all pieces is zero.
\begin{thm}[{\cite{Gromov82,BBFIPP14}}]
\label{thm:sv_additivity}
	In the situation of Setup~\ref{setup:gluing},
	suppose that every boundary component of the~$(M_j)_j$ has amenable fundamental group. 
	Then
	\[
		\|M,\partial M\|\le \sum_{j=1}^k \|M_j,\partial M_j\|.
	\]
	In particular, if $\|M_j,\partial M_j\|=0$ for all~$j\in \{1,\ldots,k\}$, then $\|M,\partial M\|=0$.
\end{thm}

We prove a vanishing result similar to Theorem \ref{thm:sv_additivity}.
While our assumptions on the pieces are stronger than the vanishing of simplicial volume, so is our conclusion on the gluing.
The advantage of our result is that the boundary components need not be amenable.

\begin{thm}
\label{thm:cat_gluing}
	In the situation of Setup~\ref{setup:gluing},
	suppose that the following hold:
	\begin{enumerate}[label=\enum]
		\item For all~$i\in \{1,\ldots,l\}$, the subspace~$S_i^\pm$ is $\pi_1$-injective in~$M_{j^\pm(i)}$;
		\item For all~$i\in \{1,\ldots,l\}$, we have~$\gd(\pi_1(S_i^+))\le n-2$;
		\item For all~$j\in \{1,\ldots,k\}$, we have $\cat_\AME(\pi_1(M_j))\le n-1$.
	\end{enumerate}
	Then $\cat_\AME(M) \le n-1$.

	Moreover, if every boundary component of the~$(M_j)_j$ is $\pi_1$-injective and has fundamental group of geometric dimension~$\le n-2$, then $\|M,\partial M\|=0$.
	\begin{proof}
		Since $\cat_\AME(M)\le \cat_\AME(\pi_1(M))$, we will prove the stronger statement $\cat_\AME(\pi_1(M))\le n-1$.
		Since~$M$ is constructed as a gluing along $\pi_1$-injective boundary components, $\pi_1(M)$ is the fundamental group of a finite graph of groups with vertex groups~$(\pi_1(M_j))_j$ and edge groups~$(\pi_1(S_i^+))_i$.
		Then the claim follows from Corollary~\ref{cor:gog}~\ref{item:gog_max}.
		The vanishing of simplicial volume follows from Lemma~\ref{lem:sv_vanishing}.
	\end{proof}
\end{thm}

While the dimension of the boundaries is $n-1$, Theorem~\ref{thm:cat_gluing} is useful in situations where the boundaries have fundamental groups of small geometric dimension.
In another direction, it is easy to construct an amenable cover of the gluing from amenable covers of the pieces and of their boundaries.
While we could apply Corollary~\ref{cor:gog}~\ref{item:gog_sum} on the level of fundamental groups, the estimate holds already on the level of spaces.

\begin{prop}
\label{prop:cat_gluing_sum}
	In the situation of Setup~\ref{setup:gluing}, suppose that~$S_i^\pm$ is $\pi_1$-injective in~$M_{j^\pm(i)}$ for all~$i\in \{1,\ldots,l\}$.
	Then
	\[
		\cat_\AME(M)
        \le \max_{j\in \{1,\ldots,k\}} \cat_\AME(M_j)
        + \max_{i\in \{1,\ldots,l\}}\bigl(\cat_\AME(S_i^+)+1\bigr).
	\]
	\begin{proof}
		The universal covering~$\widetilde{M}$ of the gluing~$M$ is a tree of spaces, which is a $\pi_1(M)$-pushout of the form
		\[\begin{tikzcd}
			\coprod_{i=1}^l \pi_1(M)\times_{\pi_1(S_i^+)}\widetilde{S_i^+}\times S^0\ar{r}\ar[hook]{d}
			& \coprod_{j=1}^k \pi_1(M)\times_{\pi_1(M_j)}\widetilde{M_j}\ar{d}
			\\
			\coprod_{i=1}^l \pi_1(M)\times_{\pi_1(S_i^+)}\widetilde{S_i^+}\times D^1\ar{r}
			& \widetilde{M}
		\end{tikzcd}\]
		Then the claim follows from Lemma~\ref{lem:catGF_basic} and Lemma~\ref{lem:key}~\ref{item:key_sum}.
	\end{proof}
\end{prop}

Basic examples of gluings are given by twisted doubles.

\begin{defn}
\label{defn:double}
	Let~$(M,\partial M)$ be an oriented compact connected manifold with non-empty boundary.
	Let~$f\colon \partial M\to \partial M$ be an orientation-preserving self-homeo\-morphism.
	The \emph{twisted double~$D_f(M)$ of~$M$ along~$f$} is the oriented closed connected manifold defined as
	\[
		D_f(M)\coloneqq M\cup_f -M,
	\]
	where~$-M$ is a copy of~$M$ with the opposite orientation.
	
	For the identity~$f=\id$ on~$\partial M$, one obtains the \emph{(untwisted) double~$D_{\id}(M)$ of~$M$}.
\end{defn}

The simplicial volume of twisted doubles satisfies subadditivity
\begin{equation}
\label{eqn:sv_double}
	\|D_f(M)\|\le 2\cdot \|M,\partial M\|
\end{equation}
if every boundary component of~$M$ has amenable fundamental group (Theorem~\ref{thm:sv_additivity}) or if $f$ is the identity on~$\partial M$ (e.g., \cite[Example~2.8]{LMR22}).
However, the inequality~\eqref{eqn:sv_double} does not hold in general~\cite[Remark~7.9]{Frigerio17}.
As applications of Theorem~\ref{thm:cat_gluing} and Proposition~\ref{prop:cat_gluing_sum}, we obtain vanishing results for the simplicial volume of twisted doubles.

\begin{cor}
\label{cor:cat_double}
	Let~$M$ be an oriented compact connected $n$-manifold with non-empty boundary.
	Suppose that the following hold:
	\begin{enumerate}[label=\enum]
		\item Every boundary component of~$M$ is $\pi_1$-injective in~$M$;
		\item Every boundary component of~$M$ has fundamental group of geometric dimension $\le n-2$;
		\item $\cat_\AME(\pi_1(M))\le n-1$.
	\end{enumerate}
	Let~$f\colon \partial M\to \partial M$ be an orientation-preserving self-homeomorphism.
	Then we have $\cat_\AME(D_f(M))\le n-1$.
	In particular, $\|D_f(M)\|=0$.
    \qed
\end{cor}

\begin{cor}
\label{cor:cat_double_sum}
	Let~$M$ be an oriented compact connected $n$-manifold with non-empty boundary.
	Suppose that every boundary component~$S_1,\ldots,S_l$ of~$M$ is $\pi_1$-injective in~$M$.
	Let~$f\colon \partial M\to \partial M$ be an orientation-preserving homeomorphism.
	Then
	\[
        \pushQED{\qed}
		\cat_\AME(D_f(M))\le \cat_\AME(M)+\max_{i\in \{1,\ldots,l\}}\bigl(\cat_\AME(S_i)+1\bigr).
        \qedhere
        \popQED
	\]
\end{cor}

We give concrete examples to which Corollary~\ref{cor:cat_double} and Corollary~\ref{cor:cat_double_sum} apply, but not Theorem~\ref{thm:sv_additivity} because the boundary is not amenable.

\begin{ex}
    	\label{ex:gluings}
    	Let~$n \in \IN$ with~$n\geq 4$.
   	Let~$V$ be an oriented compact $(n-1)$-manifold obtained by removing an open ball from an aspherical closed connected $(n-1)$-manifold.
	The boundary~$\partial V$ is an $(n-2)$-sphere.
    	Consider the $n$-manifold that is the boundary connected sum
   	\[
    		M \coloneqq (V\times S^1)\natural (V\times S^1).
   	\]
    	The boundary~$\partial M$ is homeomorphic to the connected sum $(S^{n-2}\times S^1)\# (S^{n-2} \times S^1)$ and~$\partial M$ is $\pi_1$-injective in~$M$.
    	Since~$\pi_1(\partial M)$ is a free group of rank~$2$, we have $\gd(\pi_1(\partial M))=1$.
    	Moreover, we have $\cat_\AME(\pi_1(M))\le n-1$ by Theorem~\ref{thm:main_max} and a standard estimate for~$\cat_\AME$ of products.
    	Then Corollary~\ref{cor:cat_double} shows that the simplicial volume of every twisted double of~$M$ is zero.
\end{ex}

\begin{ex}
	Let~$\Sigma_{1,1}$ denote the orientable surface of genus~$1$ with one boundary component.
    	Consider the $4$-manifold that is the boundary connected sum
   	\[
    		M\coloneqq (\Sigma_{1,1}\times S^1\times S^1)\natural (\Sigma_{1,1}\times S^1\times S^1).
    	\]
	The boundary~$\partial M$ is homeomorphic to $(S^1\times S^1\times S^1)\# (S^1\times S^1\times S^1)$ and~$\partial M$ is $\pi_1$-injective in~$M$.
    	Since $\pi_1(\partial M)\cong \IZ^3\ast \IZ^3$, we have $\cat_\AME(\partial M)\le \cat_\AME(\pi_1(\partial M))=1$.
    Since $\pi_1(M)\cong (F_2\times \IZ^2)\ast (F_2\times \IZ^2)$, we have $\cat_\AME(M)\le \cat_\AME(\pi_1(M))=1$ by Theorem~\ref{thm:main_max}.
    Then Corollary~\ref{cor:cat_double_sum} shows that the simplicial volume of every twisted double of~$M$ is zero.
\end{ex}

\section{Simplicial volume of cyclic branched coverings}
We prove a vanishing result for the simplicial volume of cyclic branched coverings.
We use the following description of cyclic branched coverings of manifolds, branched over a codimension~$2$ submanifold.

\begin{setup}
\label{setup:branching}
	Let~$n\in \IN$ with~$n\ge 3$.
	\begin{itemize}
		\item Let~$M$ be an oriented compact connected $(n-1)$-manifold with non-empty connected boundary.
		\item Let~$W$ be an oriented compact connected $n$-manifold whose boundary is homeomorphic to the untwisted double of~$M$ (Definition~\ref{defn:double}):
		\[
			\partial W\cong D_{\id}(M)=M\cup_{\id} -M.
		\]
		\item Let~$d\in \IN_{\ge 1}$.
		We denote by~$\brW{d}$ the oriented closed connected $n$-manifold obtained by gluing $d$ copies of~$W$ cyclically along copies of~$M$ via the identity map on~$M$.
		More precisely, for~$j\in \{1,\ldots,d\}$, let~$W_j=W$ and~$M_j=M$ with~$\partial W_j\cong M_j\cup_{\id} -M_j$.
		Then~$\brW{d}$ is obtained from~$\coprod_{j=1}^d W_j$ by identifying~$-M_j$ with~$M_{j+1}$, where~$j$ is considered modulo~$d$.
	\end{itemize}
\end{setup}

\begin{figure}[h]
  \def\wpiece#1#2{
    \begin{scope}[shift={#1}]
       \begin{scope}[rotate={#2}]
          \begin{scope}[shift={(0,0.2)}]
            \filldraw[fill=black!20] (0,0) -- (30:1) .. controls +(90:1.5) and +(90:1.5) .. (150:1) -- cycle;
            \fill (0,0) circle (0.1);
            \filldraw[fill=white] (0,0.8) circle (0.3);
          \end{scope}
       \end{scope}
    \end{scope}}
  \begin{center}  
  \begin{tikzpicture}[thick]
    \wpiece{(0,0)}{0}
    \draw (0,-0.2) node {$\partial M$};
    \draw (-0.7,0.2) node {$M$};
    \draw (0.7,0.2) node {$-M$};
    \draw (0,2.2) node {$W$};

    \begin{scope}[shift={(4.5,0)}]
      \wpiece{(0,0)}{0}
      \wpiece{(0,0)}{120}
      \wpiece{(0,0)}{240}
      \draw (0,2.2) node {$W_1 \sqcup W_2 \sqcup W_3$};
    \end{scope}

    \begin{scope}[shift={(9,0)}]
      \wpiece{(0,-0.2)}{0}
      \wpiece{(30:0.2)}{120}
      \wpiece{(150:0.2)}{240}
      \draw (0,2.2) node {$\brW{3}$};
    \end{scope}
  \end{tikzpicture}
  \end{center}
  
  \caption{The manifolds $W$ and $\brW{3}$, schematically}
\end{figure}

The projection map $\brW{d}\to \brW{1}$ is a branched $d$-fold covering with branching locus~$\partial M$.
One may view~$W$ as being obtained by cutting~$\brW{1}$ along~$M$.

We prove the subadditivity of simplicial volume similarly to Theorem~\ref{thm:sv_additivity}~\cite[Remark~9.2]{Frigerio17}.

\begin{prop}
        \label{prop:br_am}
	In the situation of Setup~\ref{setup:branching}, suppose that~$\pi_1(\partial W)$ is amenable and that~$\|M,\partial M\|=0$.
	Then 
	\[
		\|\brW{d}\|\le d\cdot \|W,\partial W\|.
	\]
	In particular, if $\|W,\partial W\|=0$, then $\|\brW{d}\|=0$.
\begin{proof}	
	We construct  ``small" relative fundamental cycles of~$(W,\partial W)$ whose restrictions to~$\partial W$ are compatible with the decomposition as an untwisted double.
	Let~$\varepsilon\in \IR_{>0}$.
	Since $\|M,\partial M\|=0$, there exists a relative fundamental cycle~$c\in C_{n-1}(M;\IR)$ of~$(M,\partial M)$ satisfying~$|c|_1\le \varepsilon$.
	We set~$\overline{c}$ to be the relative fundamental cycle of~$(-M,\partial(-M))$ corresponding to~$c$.
	We also view~$c$ and~$\overline{c}$ as elements in~$C_{n-1}(\partial W;\IR)$.
	Then~$c+\overline{c}\in C_{n-1}(\partial W;\IR)$ is a fundamental cycle of~$\partial W$ satisfying $|c+\overline{c}|_1<2\varepsilon$.
	
	Since~$\pi_1(\partial W)$ is amenable, by the equivalence theorem~\cite[Corollary~6.16]{Frigerio17}, there exists a relative fundamental cycle $w\in C_n(W;\IR)$ of~$(W,\partial W)$ satisfying
	\[
		|w|_1\le \|W,\partial W\|+\varepsilon \quad\text{and}\quad |\partial w|_1\le \varepsilon.
	\]
	Moreover, by amenability of~$\pi_1(\partial W)$, the singular chain complex~$C_*(\partial W;\IR)$ satisfies the uniform boundary condition in every degree~\cite{MM85}. Let $K \in \IR_{\geq 0}$ be a UBC-constant for~$C_*(\partial W;\IR)$ in degree~$n-1$. 
	Since~$\partial w$ and~$c+\overline{c}$ represent the same class in~$H_{n-1}(\partial W;\IR)$, there exists~$b\in C_n(\partial W;\IR)$ satisfying $\partial b=c+\overline{c}-\partial w$ and
	\[
		|b|_1\le K\cdot |\partial b|_1\le 3K\varepsilon.
	\]
	Then~$w'\coloneqq w+b\in C_n(W;\IR)$ is a relative fundamental cycle 
        of~$(W,\partial W)$
        satisfying $\partial(w')=c+\overline{c}$ and
	\[
		|w'|_1\le \|W,\partial W\|+\varepsilon+3K\varepsilon.
	\]
    
	For every~$j\in \{1,\ldots,d\}$, let~$W_j$ be a copy of~$W$ and let~$w'_j\in C_n(W_j;\IR)$ be the relative fundamental cycle of~$(W_j,\partial W_j)$ corresponding to~$w'$ as above.
	We also view~$w'_j$ as an element in~$C_n(\brW{d};\IR)$.
	Then $\sum_{j=1}^d w'_j\in C_n(\brW{d};\IR)$ is a fundamental cycle of~$\brW{d}$ satisfying
	\[
		\Bigl|\sum_{j=1}^d w'_j\Bigr|_1\le d\cdot \bigl(\|W,\partial W\|+\varepsilon +3K\varepsilon\bigr).
	\]
        Taking~$\varepsilon \to 0$ shows that $\|\brW{d}\| \leq d \cdot \|W,\partial W\|$.
\end{proof}
\end{prop}

We prove a vanishing result similar to Proposition~\ref{prop:br_am}.
Our assumptions are stronger than the vanishing of~$\|M,\partial M\|$ and~$\|W,\partial W\|$ but we do not require the amenability of~$\pi_1(\partial W)$.

\begin{thm}
\label{thm:cat_branching}
	In the situation of Setup~\ref{setup:branching}, let~$d\ge 4$ and suppose that the following hold:
	\begin{enumerate}[label=\enum]
		\item\label{item:pi1-inj}
		$M,-M,\partial M$ are $\pi_1$-injective in~$W$;
		\item\label{item:curvature}
		$\pi_1(M,x)\cap \pi_1(-M,x)=\pi_1(\partial M,x)$ inside $\pi_1(W,x)$ for some~$x\in \partial M$;
		\item $\gd(\pi_1(\partial M))\le n-3$;
		\item $\gd(\pi_1(M))\le n-2$;
		\item $\cat_\AME(\pi_1(W))\le n-1$.
	\end{enumerate}
	Then $\cat_\AME(\brW{d})\le n-1$.
	In particular, $\|\brW{d}\|=0$.
\end{thm}

If in addition to assumption~\ref{item:pi1-inj} also~$\partial W$ is $\pi_1$-injective in~$W$, then assumption~\ref{item:curvature} holds.

For the proof of Theorem~\ref{thm:cat_branching} we use that the fundamental group of~$\brW{d}$ is the fundamental group of a $d$-gon of groups~\cite[Example~II.12.17~(6)]{BH99}.
Therefore, we first establish the following consequence of Theorem~\ref{thm:main_max}. 

\begin{cor}
\label{cor:polygon}
	Let~$d\in \IN$ with~$d\ge 4$.
	Let~$\calQ_d$ be the poset of faces of a regular Euclidean $d$-gon with vertex set~$V$, edge set~$E$, and 2-face~$f$.
	Let~$G(\calQ_d)$ be a simple complex of groups over~$\calQ_d$ with vertex groups~$(G_v)_{v\in V}$, edge groups~$(G_e)_{e\in E}$, and $2$-face group~$G_f$.
	Suppose that for every vertex~$v\in V$ and the two edges~$e,e'\in E$ adjacent to~$v$, we have 
	\begin{equation}
	\label{eqn:curvature}
	G_e\cap G_{e'}=G_f \text{ inside } G_v.
	\end{equation}
	Let~$G$ be the fundamental group of~$G(\calQ_d)$ and let~$\calF$ be a family of subgroups of~$G$.
	Then
	\[
		\cat_\calF(G)\le \max\Bigl\{\sup_{v\in V}\catres{\calF}{G_v},\, \sup_{e\in E}\bigl(\gd(G_e)+1\bigr),\, \gd(G_f)+2\Bigr\}.
	\]
	\begin{proof}
        Assumption~\eqref{eqn:curvature} implies that the simple complex of groups~$G(\calQ_d)$ is non-positively curved~\cite[II.12.29]{BH99} and hence developable~\cite[Theorem~II.12.28]{BH99}.
		Then (the geometric realisation of) the development of~$G(\calQ_d)$ is a 2-dimensional $G$-CW-complex that is contractible (in fact $\mathrm{CAT}(0)$~\cite[p.~562]{BH99}) and whose stabilisers of $0$-cells are~$(G_v)_v$, $(G_e)_e$, and~$G_f$, stabilisers of $1$-cells are~$(G_e)_{e}$ and~$G_f$, and stabilisers of $2$-cells are~$G_f$, up to conjugation.
		Then the claim follows from Theorem~\ref{thm:main_max}.
	\end{proof}
\end{cor}

\begin{rem}
    Corollary~\ref{cor:polygon} generalises to fundamental groups of non-positively curved complexes of groups~\cite[Chapter~III.$\mathcal{C}$]{BH99}.
    One applies Theorem~\ref{thm:main} to their development, which is contractible~\cite[p.~562]{BH99}.
\end{rem}

\begin{proof}[Proof of Theorem~\ref{thm:cat_branching}]
	Since $\cat_\AME(\brW{d})\le \cat_\AME(\pi_1(\brW{d}))$, we will prove the stronger statement $\cat_\AME(\pi_1(\brW{d}))\le n-1$.
	The group~$\pi_1(\brW{d})$ is the fundamental group of a simple complex of groups over the poset~$\calQ_d$ with vertex groups equal to~$\pi_1(W)$, edge groups equal to~$\pi_1(M)$, and 2-face group~$\pi_1(\partial M)$.
	Then the claim follows from Corollary~\ref{cor:polygon}.
	The vanishing of simplicial volume follows from the vanishing theorem (Theorem~\ref{thm:sv_vanishing}).
\end{proof}

We give a concrete example to which Theorem~\ref{thm:cat_branching} applies.

\begin{ex}
    	Let~$n\in \IN$ with~$n\ge 4$.
    	Let~$(N,\partial N)$ be an oriented compact connected $n$-manifold with $\pi_1$-injective boundary.
    	Consider the boundary connected sum~$W \coloneqq N\natural N$.
    	The boundary~$\partial W$ is homeomorphic to the connected sum~$\partial N \# \partial N$. 
	Then~$\partial W$ is the untwisted double of the manifold~$M$ that is obtained from~$\partial N$ by removing an open ball.
	The boundary~$\partial M$ is an $(n-2)$-sphere.
	Then the assumptions~\ref{item:pi1-inj} and~\ref{item:curvature} of Theorem \ref{thm:cat_branching} are satisfied.
    	Moreover, we have $\cat_\AME(\pi_1(W))\leq \max\{\cat_\AME(\pi_1(N)),1\}$, $\gd(\pi_1(M))=\gd(\pi_1(\partial N))$, and $\gd(\pi_1(\partial M)) = 0$.
    	Hence, if $\cat_\AME(\pi_1(N)) \leq n-1$ and $\gd(\pi_1(\partial N))\leq n-2$, then Theorem~\ref{thm:cat_branching} shows that $\|\brW{d}\|=0$ for every~$d\in \IN$ with~$d\ge 4$.
    	For example, all assumptions are satisfied by the manifold~$N = V\times S^1$,
        where~$V$ is obtained by removing an open ball from an aspherical closed connected $(n-1)$-manifold as in Example~\ref{ex:gluings}.
\end{ex}

\section{Topological complexity}
Farber's topological complexity is defined via domains of continuity for sections of the path fibration~\cite{Farber03}.
We use an equivalent characterisation for the topological complexity of groups, phrased as the $\calF$-category for the diagonal family~\cite{FGLO19}.

\begin{defn}
	Let~$G$ be a group.
	Let~$\calD$ be the smallest family of subgroups of~$G\times G$ containing the diagonal subgroup.
	The \emph{topological complexity~$\TC(G)$ of~$G$} is defined as
	\[
		\TC(G)\coloneqq \cat_{\calD}(G\times G).
	\]
\end{defn}

As an application of Theorem~\ref{thm:main_max}, we obtain a combination theorem for~$\TC$.

\begin{thm}
\label{thm:TC}
	Let~$G$ be a group and let~$X$ be a (non-equivariantly) contractible $G$-CW-complex.
	Let~$\Sigma_0$ and~$\Sigma_{\ge 1}$ be sets of $G$-orbit representatives of all $0$-cells and all positive dimensional cells of~$X$, respectively.
	Then
	\begin{align*}
		\TC(G)\le
		\max\Bigl\{
		& \sup_{v\in \Sigma_0} \TC(G_v),
		\sup_{v,w\in \Sigma_0, v\neq w} \cd(G_v\times G_w),
		\\
		& \sup_{\sigma\in \Sigma_0\cup \Sigma_{\ge 1},\tau\in \Sigma_{\ge 1}}\bigl(\gd(G_\sigma\times G_\tau)+\dim(\sigma)+\dim(\tau)\bigr)\Bigr\}.
	\end{align*}
	\begin{proof}
		We apply Theorem~\ref{thm:main_max} to the $(G\times G)$-CW-complex~$X\times X$ and the family~$\calD$.
		The stabiliser of a cell~$\sigma\times \tau$ of~$X\times X$ is~$G_\sigma\times G_\tau$.
		For~$v,w\in \Sigma_0$ with~$v\neq w$, we use
		\begin{align*}
			\catres{\calD}{G_v\times G_w}
			&\le \cat_\TR(G_v\times G_w)
			= \cd(G_v\times G_w);
			\\
			\catres{\calD}{G_v\times G_v}
			&\le \cat_{\calD_v}(G_v\times G_v)
			= \TC(G_v).
		\end{align*}
		Here~$\calD_v$ denotes the smallest family of subgroup of~$G_v\times G_v$ containing the diagonal subgroup.
	\end{proof}
\end{thm}

As a special case of Theorem~\ref{thm:TC} for graphs of groups, we recover a result for amalgamated products of Dranishnikov--Sadykov~\cite[Theorem~16]{Dranishnikov-Sadykov19}.

\begin{cor}
\label{cor:TC}
	Let~$G$ be the fundamental group of a graph of groups with vertex groups~$(G_v)_{v\in V}$ and edge groups~$(G_e)_{e\in E}$.
	Then
	\begin{align*}
        \pushQED{\qed}
		\TC(G)\le \max\Bigl\{
		&\sup_{v\in V} \TC(G_v),\sup_{v,w\in V,v\neq w} \cd(G_v\times G_w),
		\\
		& \sup_{v\in V,e\in E}\bigl(\gd(G_v\times G_e)+1\bigr),\sup_{e,d\in E}\bigl(\gd(G_e\times G_d)+2\bigr)\Bigr\}.
        \qedhere
	\end{align*}
        \popQED
\end{cor}

For higher topological complexity~\cite{Farber-Oprea19}, results similar to Theorem~\ref{thm:TC} and Corollary~\ref{cor:TC} also hold.

\bibliographystyle{alpha}
\bibliography{newbib.bib}

\newcommand{\etalchar}[1]{$^{#1}$}
\begin{thebibliography}{CLOT03}

\bibitem[BBF{\etalchar{+}}14]{BBFIPP14}
M.~Bucher, M.~Burger, R.~Frigerio, A.~Iozzi, C.~Pagliantini, and M.~B.
  Pozzetti.
\newblock Isometric embeddings in bounded cohomology.
\newblock {\em J.~Topol.\ Anal.}, 6(1):1--25, 2014.

\bibitem[BH99]{BH99}
M.~R. Bridson and A.~Haefliger.
\newblock {\em Metric spaces of non-positive curvature}, volume 319 of {\em
  Grundlehren der mathematischen Wissenschaften}.
\newblock Springer-Verlag, Berlin, 1999.

\bibitem[Bro82]{Brown82}
K.~S. Brown.
\newblock {\em Cohomology of groups}, volume~87 of {\em Graduate Texts in
  Mathematics}.
\newblock Springer-Verlag, New York-Berlin, 1982.

\bibitem[BS25]{Babenko-Sabourau21}
I.~Babenko and S.~Sabourau.
\newblock Minimal volume entropy and fiber growth.
\newblock {\em J.~{\'E}c.\ polytech.\ Math.}, 12:481--521, 2025.

\bibitem[Cap]{Cap24}
P.~Capovilla.
\newblock On the (super)additivity of simplicial volume.
\newblock Preprint, arXiv:2306.13342, 2023.

\bibitem[CLM22]{CLM22}
P.~Capovilla, C.~L\"oh, and M.~Moraschini.
\newblock Amenable category and complexity.
\newblock {\em Algebr.\ Geom.\ Topol.}, 22(3):1417--1459, 2022.

\bibitem[CLOT03]{CLOT03}
O.~Cornea, G.~Lupton, J.~Oprea, and D.~Tanr\'e.
\newblock {\em Lusternik-{S}chnirelmann category}, volume 103 of {\em
  Mathematical Surveys and Monographs}.
\newblock American Mathematical Society, Providence, RI, 2003.

\bibitem[DS19]{Dranishnikov-Sadykov19}
A.~Dranishnikov and R.~Sadykov.
\newblock The topological complexity of the free product.
\newblock {\em Math.~Z.}, 293(1-2):407--416, 2019.

\bibitem[EG57]{EG57}
S.~Eilenberg and T.~Ganea.
\newblock On the {L}usternik-{S}chnirelmann category of abstract groups.
\newblock {\em Ann.\ of Math.~(2)}, 65:517--518, 1957.

\bibitem[Far03]{Farber03}
M.~Farber.
\newblock Topological complexity of motion planning.
\newblock {\em Discrete Comput.\ Geom.}, 29(2):211--221, 2003.

\bibitem[FGLO19]{FGLO19}
M.~Farber, M.~Grant, G.~Lupton, and J.~Oprea.
\newblock Bredon cohomology and robot motion planning.
\newblock {\em Algebr.\ Geom.\ Topol.}, 19(4):2023--2059, 2019.

\bibitem[FO19]{Farber-Oprea19}
M.~Farber and J.~Oprea.
\newblock Higher topological complexity of aspherical spaces.
\newblock {\em Topology Appl.}, 258:142--160, 2019.

\bibitem[Fri17]{Frigerio17}
R.~Frigerio.
\newblock {\em Bounded cohomology of discrete groups}, volume 227 of {\em
  Mathematical Surveys and Monographs}.
\newblock American Mathematical Society, Providence, RI, 2017.

\bibitem[Geo08]{Geoghegan08}
R.~Geoghegan.
\newblock {\em Topological methods in group theory}, volume 243 of {\em
  Graduate Texts in Mathematics}.
\newblock Springer, New York, 2008.

\bibitem[Gro82]{Gromov82}
M.~Gromov.
\newblock Volume and bounded cohomology.
\newblock {\em Inst.\ Hautes \'Etudes Sci.\ Publ.\ Math.}, (56):5--99, 1982.

\bibitem[Kue15]{Kue15}
T.~Kuessner.
\newblock Multicomplexes, bounded cohomology and additivity of simplicial
  volume.
\newblock {\em Bull.\ Korean Math.\ Soc.}, 52(6):1855--1899, 2015.

\bibitem[Li23]{Li23}
K.~Li.
\newblock Amenable covers of right-angled {A}rtin groups.
\newblock {\em Bull.\ Lond.\ Math.\ Soc.}, 55(2):978--989, 2023.

\bibitem[LLM]{LLM22}
K.~Li, C.~L\"oh, and M.~Moraschini.
\newblock Bounded acyclicity and relative simplicial volume.
\newblock To appear in {\em J.~Topol.\ Anal.}, arXiv:2202.05606, 2022.

\bibitem[LM22]{Loeh-Moraschini22}
C.~L\"oh and M.~Moraschini.
\newblock Topological volumes of fibrations: a note on open covers.
\newblock {\em Proc.\ Roy.\ Soc.\ Edinburgh Sect.~A}, 152(5):1340--1360, 2022.

\bibitem[LMR22]{LMR22}
C.~L\"oh, M.~Moraschini, and G.~Raptis.
\newblock On the simplicial volume and the {E}uler characteristic of
  (aspherical) manifolds.
\newblock {\em Res.\ Math.\ Sci.}, 9(3):Paper No.~44, 1--36, 2022.

\bibitem[LMS22]{Loeh-Moraschini-Sauer-22}
C.~L{\"o}h, M.~Moraschini, and R.~Sauer.
\newblock Amenable covers and integral foliated simplicial volume.
\newblock {\em New York J.~Math.}, 28:1112--1136, 2022.

\bibitem[L{\"u}c00]{Lueck00}
W.~L{\"u}ck.
\newblock The type of the classifying space for a family of subgroups.
\newblock {\em J.~Pure Appl.\ Algebra}, 149(2):177--203, 2000.

\bibitem[MM85]{MM85}
S.~Matsumoto and S.~Morita.
\newblock Bounded cohomology of certain groups of homeomorphisms.
\newblock {\em Proc.\ Amer.\ Math.\ Soc.}, 94(3):539--544, 1985.

\bibitem[MPSS20]{MPSS20}
E.~Mart\'inez-Pedroza and L.~J. S\'anchez~Salda{\~n}a.
\newblock Brown's criterion and classifying spaces for families.
\newblock {\em J.~Pure Appl.\ Algebra}, 224(10):106377, 17, 2020.

\bibitem[Sau09]{Sauer09}
R.~Sauer.
\newblock Amenable covers, volume and {$L^2$}-{B}etti numbers of aspherical
  manifolds.
\newblock {\em J.~Reine Angew.\ Math.}, 636:47--92, 2009.

\bibitem[Sta68]{Sta68}
J.~R. Stallings.
\newblock On torsion-free groups with infinitely many ends.
\newblock {\em Ann.\ of Math.~(2)}, 88:312--334, 1968.

\bibitem[Swa69]{Swa69}
R.~G. Swan.
\newblock Groups of cohomological dimension one.
\newblock {\em J.~Algebra}, 12:585--610, 1969.

\end{thebibliography}
	
\end{document}